\documentclass[12pt,sloppy]{article}
\usepackage{graphicx}
\usepackage{subfigure}
\usepackage{latexsym}
\usepackage{amsfonts}
\usepackage{amsthm}
\usepackage{amsmath}
\usepackage{chngpage}
\author{{\bf Maya Mohsin Ahmed }}
\title{{\LARGE {\bf The intricate labyrinth of Collatz Sequences }}}
\date{}

\newtheorem{thm}{Theorem}[section]
\newtheorem{prop}{Proposition}[section]

\newtheorem{exm}{Example}[section]

\newtheorem{corollary}{Corollary}[section]

\begin{document}     
\maketitle           

\begin{abstract}

 In a  previous article, we reduced the  unsolved problem of
 the  convergence  of Collatz  sequences,  to convergence  of
 Collatz sequences  of odd numbers, that are  divisible by 3.
 In this article, we further  reduce this set to odd numbers
 that  are congruent  to $21$  mod $24$.  We also  show that
 either the  Collatz sequence  of a given  odd number  or an
 equivalent Collatz sequence reverses  to a multiple of $3$.
 Moreover, we  construct a network composed  of Collatz sequences
 of all odd numbers.

\end{abstract}

\section{Introduction.}

Given a positive integer $A$, construct the sequence $c_i$ as follows:
\[  
\begin{array}{llll}
c_i & = & A & \mbox{if $i=0$;} \\ & = & 3c_{i-1}+1 & \mbox{if
  $c_{i-1}$ is odd;} \\ & = & c_{i-1}/2 & \mbox{if $c_{i-1}$ is even.}
\end{array} 
\]

The sequence $c_i$ is called a {\em Collatz sequence} with {\em
  starting number} $A$. The {\em Collatz conjecture} says that this
sequence will eventually reach the number 1, regardless of which
positive integer is chosen initially. The sequence gets in to an
infinite cycle of 4, 2, 1 after reaching 1.

\begin{exm} \label{Collatzeg} {\em 
The Collatz sequence of $27$ is:
\[
\begin{array}{l}
27, 82, 41, 124, 62, 31, 94, 47, 142, 71, 214, 107, 322, 161, 484,
242, 121, 364, 182, 91, 274, 137, \\ 412, 206, 103, 310, 155, 466,
233, 700, 350, 175, 526, 263, 790, 395, 1186, 593, 1780, 890, 445,
\\ 1336, 668, 334, 167, 502, 251, 754, 377, 1132, 566, 283, 850, 425,
1276, 638, 319, 958, 479, 1438, \\ 719, 2158, 1079, 3238, 1619, 4858,
2429, 7288, 3644, 1822, 911, 2734, 1367, 4102, 2051, 6154, \\ 3077, 9232,
4616, 2308, 1154, 577, 1732, 866, 433, 1300, 650, 325, 976, 488, 244,
122, 61, 184, 92, \\ 46, 23, 70, 35, 106, 53, 160, 80, 40, 20, 10, 5, 16,
8, 4, 2, 1,4,2,1,4,2,1, \dots
\end{array}
\]
} \end{exm}

For the  rest of this  article, we will ignore  the infinite
cycle  of $4,2,1$,  and  say that  a  Collatz sequence  {\em
  converges  to  $1$},  if   it  reaches  $1$.  The  Collatz
conjecture is  an unsolved conjecture  about the convergence
of a  sequence named  after Lothar Collatz.  A comprehensive
study of the Collatz conjecture can be found in \cite{lag1},
\cite{lag2}, and \cite{wir}.

 In this article, like in \cite{ahmed}, we focus on the subsequence of
 odd numbers of a Collatz sequence. This is because every even number
 in a Collatz sequence has to reach an odd number after a finite
 number of steps.  Observe that the Collatz conjecture implies that
 the subsequence of odd numbers of a Collatz sequence converges to
 $1$.
\begin{exm} \label{subseqodd27eg} {\em The subsequence of odd numbers of the Collatz sequence of $27$ in Example \ref{Collatzeg} is:
\[
\begin{array}{l}
27, 41, 31, 47, 71, 107, 161, 121, 91, 137, 103, 155, 233, 175, 263,
395, 593, 445, 167, 251, 377, 283, \\ 425, 319, 479, 719, 1079, 1619,
2429, 911, 1367, 2051, 3077, 577, 433, 325, 61, 23, 35, 53, 5, 1, 1,
\dots
\end{array}
\]
} \end{exm}

 If two Collatz sequences become equal after a finite number
 of terms, we  say they are {\em merging  sequences}. We say
 two Collatz  sequences are  {\em equivalent} if  the second
 odd number occurring in the sequences are same.

Let $A$  be an odd number. If $A >1$, write $A=2^kN+1$ such
that  $N$  is an  odd  number and  $k  \geq  1$. In  Section
\ref{2^ksection}, we prove if $k = 2$ then Collatz sequences
of  $A$ and  $N$ are  equivalent; if  $k >  2$  then Collatz
sequences of $A$ and  $2A+1$ are equivalent; and finally, if
$k=1$, then  the Collatz sequence of $A$  either merges with
the Collatz sequence of $N$ or $2A+1$.

 Let $A$ be a  positive integer. In \cite{ahmed}, we defined
 a sequence in the reverse direction defined as follows. The
 {\em reverse Collatz sequence} $r_i$ of $A$ is given by
\[
r_i =  \left \{ \begin{array}{llllllll}  A & \mbox{ if  $i =
    0$; } \\ \\

 \frac{r_{i-1}-1}{3} & \mbox{ if $r_{i-1} \equiv 1$ mod $3$ and $r_{i-1}$ is
  even}; \\ \\ 2r_{i-1} &  \mbox{ if $r_{i-1} \not \equiv 1$ mod $3$ and
    $r_{i-1}$ is even;} \\ \\ 2r_{i-1} &  \mbox{ if $r_{i-1}$ is odd}.
\end{array}
\right .
\]

 We say  that a reverse Collatz sequence  {\em converges} if
 the subsequence of odd numbers of the sequence converges to
 a multiple of  $3$. Let $A$ be an odd  number. We showed in
 \cite{ahmed} (also see  Section \ref{2^ksection}), that the
 Collatz  sequences of  $A$  and $4A+1$  are equivalent.  In
 \cite{ahmed},  we  conjectured  that  the  reverse  Collatz
 sequence of every number converges. Though, we cannot prove
 that conjecture yet, in Section \ref{reverse}, we show that
 the  reverse  Collatz  sequence  of either  $A$  or  $4A+1$
 converges.  We  also discuss  how  the  convergence of  the
 reverse Collatz  sequences of  $A$, $2A+1$, and  $4A+1$ are
 related in Section \ref{reverse}.

Given an  odd number $A$,  in Section \ref{reduce21section},
we show that the Collatz sequence of $A$ is equivalent to an
odd  number of  the form  $24t+21$, where  $t \geq  0$. This
result  reduces  the  unsolved  problem  of  convergence  of
Collatz  sequences  to  the  set  of odd  numbers  that  are
divisible by $3$ and are congruent to $21$ modulo $24$.

Finally,  in Section  \ref{networksection},  we construct  a
network  of  diagonal   arrays  which  contain  the  Collatz
sequences of all odd numbers.
\section{The coupling of the Collatz sequences of  $A$ and $2A+1$.} \label{2^ksection}

In  this section,  we  discuss the  merging  of the  Collatz
sequence of an odd number  $N$ with the Collatz sequences of
$2N+1$ and $4N+3$.

If  $A = 4^d  \times P  + 4^{d-1}+4^{d-2}  + \cdots+4  + 1$,
where $P$ is an integer and $d \geq 1$, then we say that $A$
is a {\em jump from $P$ of height $d$}. If $P \not \equiv 1$
mod 4, then $d$ is the {\em maximum height} of $A$.

\begin{exm} {\em  $53 = 4 \times
    13 +1  = 4^2 \times  3 +4  + 1$ is  a jump from  $13$ of
    height $1$ and a jump from $3$ of height $2$. }
\end{exm}

\begin{prop} \label{tailsthm} 
 Let $A$ be an odd  number and let $a_i$ denote the sequence
 of odd numbers  in the Collatz sequence of  $A$ with $a_0 =
 A$.
\begin{enumerate} 

\item \label{jumppart} Consider the sequence $c_i = 4c_{i-1}+1$ with $c_0=A$.
  Then for  any $i$, the  Collatz sequence of $A$  and $c_i$
  are equivalent.

\item \label{2kpart}  For some $r \geq 0$, $a_r \equiv 1$ mod $4$, and
\[ a_i = \frac{3a_{i-1}+1}{2}, i=1, \dots, r, \mbox{ and } a_{r+1}  = \frac{3a_r+1}{2^k}, \mbox{ for some } k \geq 2. \] 

\item \label{divby2^3part} If $A \equiv 1$ mod $4$, and if $A > 1$, then $A = 2^k
  \times N +1 $ for some  odd number $N$ and $k \geq 2$. Let
  $n_i$ denote  the odd numbers  in the Collatz  sequence of
  $N$ with $n_0 = N$.

\begin{enumerate}

\item If $k=2$, then $a_1 =  n_1$. Moreover, if $A = 4^d P +
  4^{d-1} + \cdots + 4^2 +  4 + 1$, such that $P \not \equiv
  1$ mod $4$, then
\[
a_1 = \frac{3A + 1}{4^d \times 2^i}  = \frac{3P + 1}{2^i},
\] where $i=1$ if $P$  is odd, and  
$i=0$ otherwise.

\item \label{3multpart} If $k  > 2$, write $k=2j+r$, $r=0$  or $r=1$. If $r=0$
  then $a_i  = 2^{k-2i}  \times 3^iN +  1$, for  $i=1, \dots,
  j-1$,  and the  Collatz sequence  of $A$  merges  with the
  Collatz  sequence of  $3^{j-1}N$.  On the  other hand,  if
  $r=1$, then  $a_i = 2^{k-2i}  \times 3^iN + 1$,  for $i=1,
  \dots, j$.

\end{enumerate}

\end{enumerate}
\end{prop}

\begin{proof} \hfill
\begin{enumerate}

\item This result is  proved as Corollary 2 in  \cite{ahmed}.

\item We get this  result by  applying  Theorem 2 in \cite{ahmed}.

\item  \begin{enumerate}

\item When  $k =  2$, $A =  4N+1$, therefore $a_1=  n_1$, by
  part  (\ref{jumppart}). Since $A$  is a  jump from  $P$ of
  height $d$,  Lemma 2 in  \cite{ahmed} implies $3A+1  = 4^d
  (3P+1)$. Consequently, $a_1 =  (3A + 1)/(4^d \times 2^i) =
  (3P + 1)/2^i$. Since $P \not  \equiv 1$ mod $4$, if $P$ is
  odd,  $P \equiv  3 \mbox{  mod }  4$, which  implies $3P+1
  \equiv 2$ mod $4$. Consequently, $i=1$. On the other hand,
  if $P$ is even, $3P+1$ is odd, hence $i=0$.

\item When  $k>2$ \[\frac{3A+1}{4}  = \frac{3\times 2^k  N +
  4}{4} = 3 \times 2^{k-2}N+1\] is an odd number. Hence $a_1
  = 2^{k-2} \times 3 \times  N + 1$. We repeat this argument
  to get  that if $r=0$ then  $a_i = 2^{k-2i}  \times 3^iN +
  1$,  for  $i=1, \dots,  j-1$.  Thus  $a_{j-1}  = 4  \times
  3^{j-1}N+1$.  Consequently, by applying  Part \ref{jumppart},
  we derive that the Collatz sequence of $A$ merges with the
  Collatz sequence of $3^{j-1}N$.  Similarly, we get that if
  $r=1$, then  $a_i = 2^{k-2i}  \times 3^iN + 1$,  for $i=1,
  \dots, j$.
\end{enumerate}
\end{enumerate}
\end{proof}


\begin{thm} \label{timefor2plythm}
Let $N$  be an odd  number. Let $n_0  = N$, $m_0 =  2n_0+1 =
2N+1$, and $l_0 = 2m_0+1$ = 4N+3 . Let $n_i, m_i$, and $l_i$
denote  the  subsequence  of  odd  numbers  in  the  Collatz
sequence of  $n_0, m_0,$ and $l_0$,  respectively. Then, for
some integer $r$, $n_{r+1}  = (3n_{r}+1)/2^k$ such that $k >
1$. Let $r$  be the smallest such integer.  Then, $m_{r+2} =
(3m_{r+1}+1)/2^j$ for some $j > 1$, and

\[ \begin{array}{l} 
m_i =  2n_i+1, \mbox{ for }  i \leq r, \\
 m_{{r}+1} = 2^kn_{{r}+1}+1 \\
l_i = 2m_i+1 \mbox{ for } i \leq r +1, \\
l_{r +2} = 2^jm_{r +2}+1
\end{array}
\]

If $k=2$, then $m_i = n_i$ for $i > r +1$. Otherwise, if
$k>2$ then $l_{r +2} = 4m_{r +2}+1$ and $l_i = m_i$ for $i > r+2$.
\end{thm}

\begin{proof} 
By part (\ref{2kpart})  of Proposition \ref{tailsthm}, there
exists an integer $r$, such that $n_{r +1} = (3n_{r}+1)/2^k$
where $k  > 1$.  Since $r$ is  the smallest such  number, we
have  $n_i= \frac{3n_{i-1}+1}{2}$  for  $0 \leq  i \leq  r$.
Therefore,  substituting  $ n_1  =  (3n_0+1)/2$  and $m_0  =
2n_0+1$, we get
\[ \frac{3m_0+1}{2}= \frac{3(2n_0+1)+1}{2} = 3n_0+2 = 2 \left (\frac{3n_0+1}{2} \right ) + 1 =
2n_1+1.
 \] 

Now $(3m_0+1)/2$ is odd because $2n_1+1$ is an odd number. Therefore,
$m_1 = (3m_0+1)/{2} = 2n_1 + 1$. Repeating the above argument we
prove that $m_i = 2n_i + 1$ for $i \leq r$.

Since
\[ \begin{array}{llll}
3m_{r}+1 & = & 3(2n_{r}+1)+1 & (\mbox{substituting }
m_{r} = 2n_{r}+1)\\ \\ & = & 2(3n_{r}+1)+2 &
(\mbox{rearranging terms} )\\ \\ & = & 2( 2^k n_{r+1})+2 &
(\mbox{substituting }3n_{r}+1 = 2^k n_{r+1})
\end{array}
\]

we get $(3m_{r}+1)/{2} = 2^k n_{r+1}+1$ is an odd number.  Therefore
$m_{r+1} = (3m_{r}+1)/{2} = 2^k n_{r+1}+1$.  Consequently,
$(3m_{r+1}+1)$ is divisible by $4$, hence $m_{r+2} =
(3m_{r+1}+1)/{2^j}$ such that $j > 1$.

If $k=2$, $m_{r+1} = 4  n_{r+1}+1$, hence $m_i = n_i$ for $i
> r+1$ by part (\ref{jumppart}) of Proposition \ref{tailsthm}.

Again, since $l_0 = 2m_0 + 1$ and $m_i = (3m_{i-1}+1)/{2}$ for $i
\leq r+1$, we repeat the above arguments to conclude that $l_i =
2m_i+1$ for $0 \leq i \leq r+1$.

Since

\[
\frac{3l_{r+1} + 1}{2} = \frac{3(2m_{r+1}+1) + 1}{2} = 3m_{r+1}+2,
\]

and  $3m_{r+1}+2$ is an odd number, we get 

\begin{equation} \label{cip2eqn}
l_{r+2} = \frac{3l_{r+1} + 1}{2} = 3m_{r+1}+2.
\end{equation}

Now 
\[
\frac{3m_{r+1} + 1}{4} = \frac{3(2^kn_{r+1}+1)+1}{4} = 3 \times 2^{k-2}n_{r+1}+1. 
\]

If  $k>2$, $3 \times 2^{k-2}n_{r+1}+1$ is odd. Therefore $m_{r+2} =
\frac{3m_{r+1} + 1}{4}$. Consequently,
\begin{equation} \label{bip2eqn}
4m_{r+2}+1 = 3m_{r+1}+2.
\end{equation}

From  equations  \ref{cip2eqn}  and  \ref{bip2eqn},  we  get
$l_{r+2} =  4m_{r+2}+1$. Again, by  part (\ref{jumppart}) of
Proposition \ref{tailsthm}, we get $l_{i} = m_i$,  for $i >
r+2$.
\end{proof}

\begin{exm} {\em  We illustrate Theorem \ref{timefor2plythm} with an example.
Let $A=3$, $n_0=A=3$ $m_0= 2n_0+1=7$, and $l_0=2m_0+1 = 15$. Let $n_i, m_i, l_i$ denote the odd numbers in the Collatz sequence of $n_0$, $m_0$, and $l_0$, respectively. Thus 
\[
\begin{array}{l}
n_i: 3, 5, 1 \\
m_i: 7, 11, 17, 13, 5, 1 \\
l_i: 15, 23, 35, 53, 5, 1 
\end{array}
\]

Since    $n_2    =    \frac{3n_1+1}{2^4}=1$,   by    Theorem
\ref{timefor2plythm},
\[\begin{array}{l}
m_i = 2n_i+1  \mbox{ for  } 0 \leq i \leq 1, \\
m_2 = 2^4n_2+1 = 17 \\
l_i = 2m_i+1 \mbox{ for } 0 \leq i \leq 2, \\
l_3 = 4m_3+1 = 53, \mbox{ and } \\
l_i = m_i \mbox{ for } i \geq 4.
\end{array}
\]

Thus, in this case, the Collatz sequence of $15$ merges with
the Collatz sequence of $7$. } \end{exm}

We  conclude this  section  by observing  the following.  If
$A>1$ is an odd number,  then write $A=2^kN+1$, such that, $N$
is an odd  number and $k \geq  1$. If $k = 2$  then, by Part
\ref{jumppart}  of Proposition  \ref{tailsthm},  the Collatz
sequences of  $A$ and $N$ are  equivalent. If $k  > 2$, then
the odd  number $p$  that occurs after  $A$, in  the Collatz
sequence of  $A$, is $p=(3A+1)/4$. This  implies, by Theorem
\ref{timefor2plythm},   that  the  number   occurring  after
$2A+1$,  in  the  Collatz  sequence of  $2A+1$,  is  $4p+1$.
Consequently, the  Collatz sequences  of $A$ and  $2A+1$ are
equivalent.   Finally,   if    $k=1$,   then,   by   Theorem
\ref{timefor2plythm},  the Collatz  sequence  of $A$  either
merges with the Collatz sequence of $N$ or $2A+1$.


\section{The collaborative convergence  of the reverse Collatz sequences of $A$, $4A+1$, and $2A+1$. } \label{reverse}

Let $A$ be an odd number.  In this section, we show that, if
$A  \equiv  0$ mod  $3$,  then  either  the reverse  Collatz
sequence of $2A+1$  or $4A+1$ converge. If $A  \equiv 1$ mod
$3$, then  $2A+1 \equiv 0$  mod $3$, and either  the reverse
Collatz sequence  of $A$ or $4A+1$ converge.  Finally, If $A
\equiv 2$ mod $3$, then  $4A+1 \equiv 0$ mod $3$, and either
the reverse Collatz sequence of $A$ or $2A+1$ converge.

We  say that  the reverse  Collatz sequence  of $A$  is {\em
  trivial}  if $A$  is the  only odd  number in  its reverse
Collatz sequence. We now record some properties of a reverse
Collatz sequence given in Proposition 1 of \cite{ahmed}.

\begin{prop} \label{somerevcollprop}
Let $A$ be an odd number.
\begin{enumerate}
\item \label{trivialrevcollpart} If  $A$ is divisible by $3$
  then the reverse Collatz sequence of $A$ is trivial.

\item \label{reversepicompart}
 Let $p_i$ denote the subsequence of odd
  numbers in the reverse Collatz sequence of $A$. If $p_i \equiv 0$
  mod $3$, then $p_{i+1}$ do not exist. Otherwise, $p_{i+1}$ is the
  smallest odd number before $p_i$ in any Collatz sequence and
\begin{equation} \label{reversecollapieqn}
p_{i+1} = \left \{ \begin{array}{l}
\frac{2p_i-1}{3} \mbox{ if } p_i \equiv 2 \mbox{ mod } 3 \\ \\
\frac{4p_i-1}{3} \mbox{ if } p_i \equiv 1 \mbox{ mod } 3 
\end{array}
\right .
\end{equation}

\item The only odd number that can be a jump from another odd number
  in a reverse Collatz sequence is the starting number of the reverse
  Collatz sequence.

\end{enumerate}
\end{prop}

\begin{corollary} \label{howtocomoutepiandjumppart}
 Let $A$ be an odd number and let $B = 2^kA+1$ such that $k
 \geq 1$.  Let $p_i$ denote  the sequence of odd  numbers in
 the reverse  Collatz sequence of $p_0=B$. Then
  
\[ p_{1} = \left \{ \begin{array}{ll}
2^{k+2} (\frac{A}{3}) +1 & \mbox{ if } A  \equiv 0
 \mbox{ mod } 3; \\ \\
 2( \frac{2^kA-1}{3})+1 & \mbox{ if } A \not \equiv 0
 \mbox{ mod } 3.
\end{array}
\right .
\]

\end{corollary}

{\em Proof.}

If $A  \not \equiv 0$  mod $3$, consider the  four following
situations: $k$ is even and  $A \equiv 1$ mod $3$ implies $B
\equiv 2$  mod $3$;  $k$ is  even and $A  \equiv 2$  mod $3$
implies $B  \equiv 0$ mod $3$;  $k$ is odd and  $A \equiv 1$
mod $3$  implies $B  \equiv 0$  mod $3$; $k$  is odd  and $A
\equiv 2$ mod $3$ implies $B \equiv 2$ mod $3$.

By    part     (\ref{reversepicompart})    of    Proposition
\ref{somerevcollprop}, $p_1$  exists only if  $B \not \equiv
0$ mod $3$. Consequently, when $B \equiv 2$ mod $3$, by part
(\ref{reversepicompart})            of           Proposition
\ref{somerevcollprop},
\[    p_{1}    =   \frac{2B-1}{3}    =
\frac{2^{k+1}A  +   1}{3}  =  \frac{2^{k+1}A-2+3}{3}   =  2(
\frac{2^kA-1}{3})+1. \]

 Finally, if  $A \equiv  0$ mod $3$,  then $B \equiv  1$ mod
 $3$.   Therefore,  by   part   (\ref{reversepicompart})  of
 Proposition \ref{somerevcollprop},

\[ p_1 = \frac{4B-1}{3} =   \frac{2^{k+2}A + 3}{3} = 2^{k+2} (\frac{A}{3}) +1.\]

\qed

\begin{thm} \label{multiply2reversetill4thm}
Let $A$  be an  odd number and  let $p_i$, $q_i$,  and $f_i$
denote the subsequence of odd numbers in the reverse Collatz
sequence of $A$, $2A+1$, and $4A+1$, respectively.

\begin{enumerate}
 
\item If $A \equiv 0$ mod $3$, write $A = 3^nB$ such that $n
  \geq 1$, and $B$ is an odd number such that $B \not \equiv
  0$    mod   $3$.    Then,   for    $i=1,2,    \dots,   n$,
  $q_i=2^{2i+1}(3^{n-i}B)+1$ and $f_i = 2^{2i+2}(3^{n-i}B) +
  1$. If $B \equiv 1$ mod  $3$, then $q_n \equiv 0$ mod $3$.
  On  the other hand,  if $B  \equiv 2$  mod $3$,  then $f_n
  \equiv 0$ mod $3$.

\item  Let $A  \not \equiv  0$ mod  $3$. Then  the following
  conditions are true.

\begin{enumerate}

\item \label{notmod3part} For some  $k \geq 1$,  $p_k \not \equiv 2$ mod
  $3$.  If  $p_k \equiv 1$  mod $3$, then $p_k  =  3n+1$  and
  $p_{k+1} = 4n+1$ for some even number $n$. 

\item \label{4aandapart} If $A \equiv 1$  mod $3$, the reverse Collatz sequence
  of  $2A+1$ is  trivial. The  reverse Collatz  sequences of
  both $A$ and  $4A+1$ are not trivial, and  $f_i = 2p_i+1$,
  for $1  \leq i \leq  k$. If $p_k  \equiv 0$ mod  $3$, then
  $f_{k+1} =  2^3 \left (  \frac{p_{k}}{3} \right )  +1$. On
  the  other hand,  if $p_k  \equiv  1$ mod  $3$, then  $f_k
  \equiv 0$ mod $3$.

\item \label{deffiinas2pipart} Finally,  if $A \equiv 2$ mod
  $3$, then the reverse  Collatz sequence of $4A+1 \equiv 0$
  is trivial. The reverse  Collatz sequences of both $A$ and
  $2A+1$ are not trivial, and  $q_i = 2p_i+1$, for $0 \leq i
  \leq k$.  If $p_k \equiv 0$  mod $3$, then  $q_{k+1} = 2^3
  \left (  \frac{p_{k}}{3} \right ) +1$. On  the other hand,
  if $p_k \equiv 1$ mod $3$, then $q_k \equiv 0$ mod $3$.
\end{enumerate}
\end{enumerate}
\end{thm}
 
\begin{proof}  \hfill

\begin{enumerate}

\item If $A \equiv 0$ mod $3$, since $q_0=2A+1 \equiv 1$ mod
  $3$  and $f_0  =  4A+1  \equiv 1$  mod  $3$, by  Corollary
  \ref{howtocomoutepiandjumppart}, we get $q_1 = 2^3 \left (
  \frac{A}{3} \right ) + 1 = 2^3(3^{n-1}B)+1$ and $f_1 = 2^4
  \left  ( \frac{A}{3}  \right )  + 1  =  2^4(3^{n-1}B)+1 $.
  Observe  that,  if  $n  >1$,  both  $q_1$  and  $f_1$  are
  equivalent  to $1$  mod  $3$. Hence,  again, by  Corollary
  \ref{howtocomoutepiandjumppart},    we    get    $q_2    =
  2^5(3^{n-2}B)+1$ and $f_1  = 2^6(3^{n-2}B)+1 $. Continuing
  the  same   argument,  we  get  for   $i=1,2,  \dots,  n$,
  $q_i=2^{2i+1}(3^{n-i}B)+1$ and $f_i = 2^{2i+2}(3^{n-i}B) +
  1$. Now, if $n$ is odd and $B \equiv 1$ mod $3$, then $q_n
  = 2^{2n+1}B +1 \equiv 0$ mod  $3$. If $n$ is even and if
  $B \equiv  2$ mod $3$,  then $f_n = 2^{2n+2}B+1  \equiv 0$
  mod $3$.

\item Let $A \not \equiv 0$ mod $3$.
\begin{enumerate} 

\item  If   $p_i  \equiv  2$   mod  $3$,  then   $p_{i+1}  =
  (2p_i-1)/3$, by part \ref{reversepicompart} of Proposition
  \ref{somerevcollprop}. Therefore  $p_{i+1} < p_i$.  If all
  the odd terms  of the reverse Collatz sequence  of $A$ are
  equivalent to  $2$ mod $3$  then the sequence  is strictly
  decreasing.  The  smallest such  odd  number  that such  a
  sequence can reach is $5$,  in which case, the next number
  in the sequence is $3$, a contradiction. Hence, there is a
  $k$ such that $p_k \not  \equiv 2$ mod $3$. If $p_k \equiv
  1$ mod $3$,  and since $p_k$ is odd,  $p_k =3n+1$ for some
  even number $n$.  Moreover, by part \ref{reversepicompart}
  of Proposition \ref{somerevcollprop}, $p_{k+1} = (4p_k-1)/3
  = 4n+1$.

\item  By  Part  (\ref{trivialrevcollpart})  of  Proposition
  \ref{somerevcollprop},  the  reverse  Collatz sequence  is
  trivial if the starting number is a multiple of $3$. If $A
  \equiv 1$ mod  $3$, then $2A+1 \equiv 0$  mod $3$ and $4A+1
  \equiv  2$  mod   $3$.  Therefore,  the  reverse  Collatz
  sequence  of  $2A+1$  is  trivial,  whereas,  the  reverse
  Collatz sequences of both  $A$ and $4A+1$ are not trivial.
  By    part    (\ref{reversepicompart})   of    Proposition
  \ref{somerevcollprop}, $p_1 =(4A-1)/3$, and
  $f_1 =  (2f_0-1)/3 =  (2(4A+1)-1)/3 = (8A+1)/3$.  Hence \[
  2p_1+1  =  2  \left  (  \frac{4A-1}{3}  \right  )  +  1  =
  \frac{8A+1}{3} = f_1.
\]

We  use induction  to prove  that  $f_i =  2p_i+1$, if  $p_i
\equiv 2$,  for $i > 1$.  Assume the hypothesis  is true for
$i=r$, that  is $p_r \equiv 2$  mod $3$ and  $f_r = 2p_r+1$.
Consequently,  $f_r \equiv  2$ mod  $3$. Therefore,  by part
\ref{reversepicompart} of Proposition \ref{somerevcollprop},
$p_{r+1} =  (2p_r -1)/3$ and  $f_{r+1} = (2f_r  -1)/3$. This
implies

\[
2p_{r+1}+1 = 2 \left ( \frac{2p_r-1}{3} \right )+1 = \frac{4p_r+1}{3}
= \frac{2(2p_r+1)-1}{3} = \frac{2f_r-1}{3} = f_{r+1}.
\]

Therefore, it follows by  induction that $f_i = 2p_i+1$ for
$1 \leq i \leq k$.

If $p_k \equiv 0$ mod $3$, then since $f_k = 2p_k+1$, we get
by Corollary \ref{howtocomoutepiandjumppart} that $f_{k+1} =
2^3 \left  ( \frac{p_{k}}{3} \right ) +1$.  Clearly, if $p_k
\equiv 1$ mod $3$, then $f_k \equiv 0$ mod $3$.

\item If $ A \equiv 2$ mod $3$, then $q_0=2A+1 \equiv 2$ mod
  $3$ and  $4A+1 \equiv 0$  mod $3$. Therefore,  the reverse
  Collatz  sequence  of  $4A+1$  is  trivial,  whereas,  the
  reverse Collatz  sequences of both $A$ and  $2A+1$ are not
  trivial.  By part (\ref{reversepicompart})  of Proposition
  \ref{somerevcollprop},  $p_1  =   (2A-1)/3$,  and  $q_1  =
  (2q_0-1)/3 =  (2(2A+1)-1)/3 = (4A+1)/3$.  Consequently, \[
  2p_1+1  =  2  \left  (  \frac{2A-1}{3}  \right  )  +  1  =
  \frac{4A+1}{3} = q_1.
\]
We repeat the  induction argument in part (\ref{4aandapart})
to  get $q_i  =2p_i+1$  for $0  \leq  i \leq  k$. Again,  by
Corollary \ref{howtocomoutepiandjumppart}, if $p_k \equiv 0$
mod $3$,  since $q_k = 2p_k+1$,  we get that  $q_{k+1} = 2^3
\left ( \frac{p_{k}}{3} \right )  +1$. On the other hand, if
$p_k \equiv 1$ mod $3$, then $q_k \equiv 0$ mod $3$.
\end{enumerate}
\end{enumerate}
\end{proof}


\section{Reduction of the Collatz conjecture   to  odd numbers that are congruent to $21$ mod $24$.} \label{reduce21section}
 Let $A$ be an odd number. In this section, we show that the
 Collatz sequence of $A$ is  equivalent to an integer of the
 form $24t+21$ where $t \geq 0$.

\begin{thm} \label{12a+pthm}

Let $A>1$ be  an odd number and let  $a_i$ denote the sequence
of odd  numbers in the Collatz  sequence of $A$  with $a_0 =
A$.  Construct  a  sequence  $b_i$ such  that

 \[b_0 = \left \{ \begin{array}{llllll}
12 \left ( 2^4\left (\frac{A}{3} \right )+1 \right ) + 9  & \mbox{if } A \equiv 0 \mbox{ mod } 3; \\ \\
12 \left ( \frac {4A-1}{3}   \right ) + 9 & \mbox{if } A \equiv 1 \mbox{ mod } 3;
 \\ \\
12 \left ( \frac{2A-1}{3}   \right ) + 9 & \mbox{if } A \equiv 2 \mbox{ mod } 3.
\end{array} \right .
\] 
 If for any $i$, $a_i =1$,  then $b_i=12 \times 1 + 9 = 21$.
 Let $i \geq 1$ and let $a_{i-1} > 1$. If $a_{i-1} \equiv 3$
 mod  $4$ then  let $b_i  = 12  \left  ( \frac{a_{i-1}-1}{2}
 \right  )+9$.   If  $a_{i-1}   \equiv  1$  mod   $4$,  then
 $a_{i-1}=2^kN+1$ for  some odd number $N$, and  $k \geq 2$.
 If $k >2$, then $b_i=12a_{i-1}+9$. If $k=2$, then $a_{i-1}$
 is a jump  from an odd number $P$, such that,  $P$ is not a
 jump from an odd number. If  $P=1$, then $b_i = 12 \times 1
 + 9=  21$. If $P>1$, write  $P=2^rq+1$ such that  $q$ is an
 odd number. Then, $r \not = 2$. If $r=1$, let $b_i=12 \left
 ( \frac{P-1}{2} \right ) +9$. Otherwise, let $b_i=12P+9$.

Then, for all $i \geq 0$,
\begin{enumerate}
\item  $b_i$ is  odd, $b_i$  is divisible  by $3$,  and $b_i
  \equiv 9$ mod $12$.
\item the  Collatz sequences of $a_i$ and $b_i$ are equivalent.
\item if $a_i > 1$,  write  $a_i = 2^k  n_i + 1$  such that $k \geq  1$ and
  $n_i$ is an odd number. If $k \not = 2$, then
\[
b_{i+1} = \left \{ \begin{array}{lllllll} 3(2a_i+1), & \mbox{ if $k=1$,} \\ \\
3(4a_i+3), & \mbox{ otherwise.}
\end{array} \right .
\]

If $k=2$,  then $a_i$  is a jump  from an odd  number $j_i$,
such  that, $j_i$  is  not a  jump  from an  odd number.  If
$j_i>1$,  write $j_i=2^rq_i+1$  such  that, $r  \geq 1$  and
$q_i$ is an odd number. Then $r \not = 2$, and

\[
b_{i+1} = \left \{ \begin{array}{lllllll} 3(2j_i+1), & \mbox{ if $r=1$,} \\ \\
3(4j_i+3), & \mbox{ otherwise.}
\end{array} \right .
\]

\end{enumerate}
\end{thm}

\begin{proof}
\begin{enumerate}
\item It follows  by definition that $b_i$ is  odd, $b_i$ is
  divisible by $3$, and $b_i \equiv 9$ mod $12$.

\item When  $A \equiv 0$ mod  $3$, let $p_1$  denote the odd
  number  that occurs  after $4A+1$  in the  reverse Collatz
  sequence     of      $4A+1$.     Then     by     Corollary
  \ref{howtocomoutepiandjumppart},     $p_1    =    2^4\left
  (\frac{A}{3} \right )+1$. We see that
\[
b_0=12p_1+9 =12 \left ( 2^4\left (\frac{A}{3} \right )+1 \right ) + 9 =   4^3A + 4^2+4+1.
\]
Thus, $b_0$ is a jump of height $3$ from $A$. 

Now, let $p_1$ denote the  odd number that comes after $A$ in
the  reverse  Collatz  sequence  of $A$.  Then  by  Equation
(\ref{reversecollapieqn}), we get

\begin{equation*} 
p_{1} = \left \{ \begin{array}{l}
\frac{2A-1}{3}, \mbox{ if } A \equiv 2 \mbox{ mod } 3 \\ \\
\frac{4A-1}{3}, \mbox{ if } A \equiv 1 \mbox{ mod } 3 
\end{array}
\right .
\end{equation*}

Like before,  we see that  $12p_1+9=b_0$ which is a  jump of
height of  $2$ from $A$,  when $A \equiv  1$ mod $3$,  and a
jump of height of $1$ from $A$, when $A \equiv 2$ mod $3$.

Now let  $i \geq 1$. When  $a_{i-1} \equiv 3$  mod $4$, then
$a_i=\frac{3a_{i-1}+1}{2}$. Consequently, 
\[
b_i = 12  \left (  \frac{a_{i-1}-1}{2} \right )+9 = 6a_{i-1}+3= 4a_i+1.
\]
Thus, $b_i$ is a jump of height $1$ from $a_i$.

Finally,  let   $a_{i-1}  \equiv  1$  mod   $4$.  If  $k>2$,
$a_i=\frac{3a_{i-1}+1}{4}$.

Then $4^2a_i+4+1 = 12a_{i-1}+9 = b_i$ is a jump from $a_i$ of height $2$.

When $k=2$  the Collatz sequence of  $a_{i-1}$ is equivalent
to the Collatz sequence of  $P$. By the definition of $P$, $r
\not =  2$. If $r>2$, then  $a_i = (3P+1)/4$,  and if $r=1$,
$a_i = (3P+1)/2$. We repeat  the above argument and get that
$b_i$ is  a jump of height  $2$ from $a_i$  when when $k>2$,
and a jump of height $1$, otherwise.

Thus, for  any $i$, $b_i$ is  a jump from  $a_i$. Hence, the
Collatz sequences of $a_i$ and $b_i$ are equivalent.
\item Since  $a_i$ is odd,  when $k=1$, $a_i \equiv  3$ mod
  $4$. Therefore
\[b_{i+1}= 12 \left ( \frac{a_i-1}{2} \right ) + 9 = 3(2a_i + 1).\] 
When   $k>2$,   $a_i   \equiv   1$   mod   $4$.   Therefore,
$b_{i+1}=12a_i+9=3(4 a_i+3)$.

Let $k=2$, then $r \not =2$ by definition of $j_i$, and
\[
b_{i+1}  =  \left   \{  \begin{array}{lllllll}  12  \left  (
  \frac{j_i-1}{2}  \right )  + 9  = 3(2j_i+1),  &  \mbox{ if
    $r=1$,} \\ \\ 12j_i+9 = 3(4j_i+3), & \mbox{ otherwise.}
\end{array} \right .
\]

\end{enumerate}
\end{proof}

\begin{corollary}
\begin{enumerate}
\item Given an  odd number $a$, the Collatz  sequence of $a$
  is equivalent to an integer  of the form $24t+21$ where $t
  \geq 0$.
\item Consider a number of  the form $24t+21$, where $t \geq
  0$. Let  $p=2t+1$. If $t=0$, the the  Collatz sequences of
  $24t+21$  and $1$  merge. If  $t >  0$,  write $p=2^kn+1$,
  where $k \geq 1$, and $n$  is an odd number. If $k>2$, the
  Collatz  sequences of $24t+21$  and $p$  merge. Otherwise,
  the Collatz sequences of $24t+21$ and $2p+1$ merge.
\end{enumerate}
\end{corollary}

\begin{proof}
\begin{enumerate}
\item  Applying  Theorem  \ref{12a+pthm},  we get  that  the
  Collatz  sequence  of $a$  is  equivalent  to the  Collatz
  sequence of an integer of the form $12d+9$ where $d$ is an
  odd  number. Write  $d=2t+1$ such  that $t  \geq  0$. This
  gives us the desired result.
  
\item  This  result  is  a  direct  application  of  Theorem
  \ref{12a+pthm}.
\end{enumerate}
\end{proof}

\begin{exm} {\em
Let $a_i$ denote the odd  numbers in the Collatz sequence of
$319$, with $a_0=319$. For  each $i$, if $a_i>1$, write $a_i
= 2^kn_i+1$,  where $n_i$  is an odd  number. If  $k>2$ then
$e_i=4a_i+3$.  If $k=1$, then  $d_i=2a_i+1$. If  $k=2$, then
let $p_i$  be an odd number  such that, $p_i$ is  not a jump
from an  odd number. If $p_i=1$, then  $e_i=4 \times 1+3=7$.
If $p_i  > 1$, write $p_i  =2^rq_i+1$ such that  $q_i$ is an
odd number.  Then, if $r=1$, then $d_i=2q_i+1$  and if $r>1$
then  $e_i=4p_i+3$.  Let  $b_i$  be defined  as  in  Theorem
\ref{12a+pthm}. We get the  following table from the Collatz
sequence of $319$.
\[
\begin{array}{cccccccccccccccc}
b_i & a_i & d_i & 3d_i & e_i & 3e_i   \\
5109 = 12 \times 425+ 9  & 319 = 2 \times 159 + 1  & 639 & 1917\\ 
1917 = 12 \times 159 + 9 & 479 = 2 \times 239 + 1 & 959 &  2877 \\
 2877 = 12 \times 239+ 9 & 719 = 2 \times 359 + 1  & 1439 & 4317\\
4317 = 12 \times 359+ 9  & 1079 = 2 \times 539 + 1 & 2159 & 6477\\ 
6477= 12 \times 539+ 9 &  1619 = 2 \times 809 + 1  & 3239 & 9717 \\
9717= 12 \times 809+ 9 & 2429 = 4 \times 607 + 1 & 1215&  3645 \\
3645= 12 \times 303 + 9 &  911 = 2 \times 455 + 1 & 1823 & 5469 \\
5469 = 12 \times 455 + 9   &1367 = 2 \times 683 + 1 & 2735 & 8205  \\ 
8205 = 12 \times 683 + 9  & 2051 = 2 \times 1025 + 1 & 4103 & 12309  \\ 
12309 = 12 \times 1025 + 9  & 3077 = 4 \times 769 + 1 & & &  3079 & 9237 \\
9237 = 12 \times 769 + 9  & 577 = 2^6 \times 9 + 1 &   &    & 2311 & 6933 \\
 6933 = 12 \times 577 + 9  &  433 = 2^4 \times 27 + 1 & & & 1735 & 5205\\
 5205 = 12 \times 433 + 1 & 325 = 4 \times 81 + 1 &  & & 327 & 981  \\
981 = 12 \times 81 + 9  &  61 = 4 \times 15 + 1 & 31 & 93 \\ 
93 = 12 \times 7 + 9  &  23 = 2 \times 11 + 1 & 47 & 141  \\ 
 141 = 12 \times 11 + 9 &  35 = 2 \times 17 + 1 & 71  & 213   \\
213 = 12 \times 17 + 9  & 53 = 4^2 \times 3 + 4 + 1  & 7 & 21 \\ 
21= 12 \times 1 + 9  & 5 = 4 \times 1 + 1 &   & & 7 & 21 \\ 
21 = 12 \times 1 + 9  & 1
\end{array}
\]
} \end{exm}


\section{Network of Collatz sequences.} \label{networksection}

In this  section, we construct a network  of diagonal arrays
which contain the Collatz sequences of all odd numbers.

\begin{thm} \label{networkthm} 
For $n  \not \equiv  1$ mod $3$,  define a diagonal  array as
follows. Let  $u_0 =  4n+1$ and $u_i  = 2u_{i-1}+1$,  for $i
\geq 1$. For $j \geq  0$, let $v_{0,j}=u_j$, and for $k \geq
1$, let $v_{k,k} = 3v_{k-1,k-1}+2$. Finally, for $j >i$, let
$v_{i,j} = 2v_{i,j-1}+1$. We get an array
 \[\begin{array}{lllllllllllllllll}
u_0 & u_1 & u_2 &u_3 & u_4 & u_5 &  u_6 & u_7 & \dots \\
& v_{1,1} & v_{1,2} & v_{1,3} & v_{1,4} & v_{1,5} & v_{1,6} & v_{1,7} & \dots \\
&& v_{2,2} & v_{2,3} &  v_{2,4} & v_{2,5} & v_{2,6} & v_{2,7} & \dots \\
&&& v_{3,3} & v_{3,4} &  v_{3,5} & v_{3,6} & v_{3,7} &  \dots \\
&&&& \dots & \dots  \\

\end{array}
\]
with the following properties:

\begin{enumerate}

\item \label{partu_0multi3}  $u_k \not  \equiv 2$  mod $3$  for all  $k \geq 0$, whereas,
  $v_{i,j} \equiv 2$ mod $3$, if $i>0$ and $j>0$.

\item \label{partnotjump} $u_0 \equiv 1$ mod  $4$ and $v_{i,i} \equiv 1$ mod $4$
  for all  $i$. $u_i  \not \equiv 1$  mod $4$ for  $i>0$ and
  $v_{i,j} \not \equiv 1$ mod $4$ if $i \not = j$.

\item  For  $i \geq  0$,  if $u_i  \equiv  1$  mod $3$  then
  $u_{i+1} \equiv  0$ mod  $3$. On the  other hand,  if $u_i
  \equiv 0$ mod $3$ then $u_{i+1} \equiv 1$ mod $3$.

\item If $u_0  \equiv 0$ mod $3$, then write  $u_0 = 3N$ for
  some $N$. For $i \geq 0$, let $j_i$ denote the sequence of
  jumps of  height $i$  from $N$, then  $u_{2i} =  3j_i$ and
  $u_{2i+1} =  6j_i+1$. Alternately,  if $u_1 \equiv  0$ mod
  $3$, then write  $u_1 = 3N$ for some $N$.  For $i \geq 0$,
  let $j_i$ denote the sequence  of jumps of height $i$ from
  $N$. Then, $u_{2i+1} = 3j_i$ and $u_{2i}=6j_i+1$.

\item For $j \geq 1$, the $j$-th column is the first few odd
  numbers in the Collatz sequence of $u_j$.

\item  For  $j \geq  i >0$,
   $v_{i,j} = 3v_{i-1,j-1}+2$.

\item \label{v_ijgoesto1mod3part}  Let  $a_{i,i}$ denote the odd number  that comes after
  $v_{i,i}$ in the Collatz  sequence of $v_{i,i}$. If $n$ is
  even, then  for $k \geq  0$, $a_{2k,2k} \equiv 1$  mod $3$
  and $v_{2k+1,2k+1}=4a_{2k,2k}+1$. Moreover, $a_{2k,2k}$ is
  smaller than both $v_{2k,2k}$ and $v_{2k+1,2k+1}$. If $n$ is
  odd, then  for $k \geq  1$, $a_{2k-1,2k-1} \equiv 1$  mod $3$
  and $v_{2k,2k}=4a_{2k-1,2k-1}+1$. Like before, $a_{2k-1,2k-1}$ is
  smaller than both $v_{2k-1,2k-1}$ and $v_{2k,2k}$.

\item  If $n$  is  odd, then  for  $k \geq  1$, the  Collatz
  sequences  of $u_{2k-1}$  and  $u_{2k}$ merge.  If $n$  is
  even,  then  for $k  \geq  0$,  the  Collatz sequences  of
  $u_{2k}$ and $u_{2k+1}$ merge.

\item For  every $n  >0$, such that,  $n \not \equiv  1$ mod
  $3$, the Collatz sequence of $u_0$ merges with the Collatz
  sequence of a smaller number  $a$ such that $a \not \equiv
  2$ mod $3$.

\end{enumerate}

\end{thm}

\begin{proof}
\begin{enumerate}

\item Since $n \not \equiv 1$ mod $3$, $u_0=4n+1 \not \equiv
  2$ mod $3$. If $u_i  \not \equiv 2$ mod $3$, then $u_{i+1}
  = 2u_i+1  \not \equiv 2$ mod $3$.  Consequently, $u_k \not
  \equiv  2$ mod  $3$  for all  $k  \geq 0$.  

For $i >  0$, $v_{i,i} \equiv 2$ mod  $3$, by definition. If
$v_{i,j}  \equiv 2$  mod $3$,  then $v_{i,j+1}  = 2v_{i,j}+1
\equiv 2$ mod $3$. Thus,  it follows that $v_{i,j} \equiv 2$
mod $3$, if $i>0$ and $j>0$.

\item    $u_0   \equiv    1$   mod    $4$    by   definition.
  $v_{i,i}=3v_{i-1,i-1}+2 \equiv 1$  mod $4$, if $v_{i-1,i-1}
  \equiv 1$ mod $4$. Since  $v_{0,0} = u_0$, it follows that
  $v_{i,i} \equiv 1$ mod $4$ for $i \geq 0$. For $i>0$, $u_i
  =2u_{i-1}+1 \not  \equiv 1$ mod $4$ since  $u_{i-1}$ is an
  odd number. Similarly,  $v_{i,j} =2v_{i,j-1}+1 \not \equiv
  1$ mod $4$, when $i \not = j$.

 \item  Since $u_{i+1}=2u_i+1$,  if $u_i  \equiv 0$  mod $3$,
  then $u_{i+1} \equiv 1$ mod $3$. Similarly, if $u_i \equiv
  1$ mod $3$, then $u_{i+1} \equiv 0$ mod $3$

\item We prove this result by using the induction principle.
  If $u_0  = 3N$, then $u_1  = 6N+1$. Thus $u_0  = 3j_0$ and
  $u_1=6j_0+1$.  Assume  $u_{2i}  =  3j_i$ and  $u_{2i+1}  =
  6j_i+1$. Then, $u_{2(i+1)=2i+2}  = 2u_{2i+1}+1 = 3(4j_i+1)
  =  3j_{i+1}$.  And   $u_{2(i+1)+1=2i+3}  =  2u_{2i+2}+1  =
  6j_{i+1}+1$. Hence  the result.  The proof is  similar for
  the case  when $u_1 \equiv 0$  mod $3$.

\item  When $j  \geq 1$,  $v_{0,j}=u_j$. We  know  from Part
  \ref{partnotjump} that  $v_{i,j} \equiv  3$ mod $4$  if $i
  \not  = j$.  Therefore, the  odd number  that  comes after
  $v_{i,j}$  in   the  Collatz  sequence   of  $v_{i,j}$  is
  $(3v_{i,j}+1)/2$. We  do not compute the  odd numbers that
  come  after $v_{i,i}$  in  this part.  Hence, assume  that
  $j>i$.

For $k  \geq 1$, let $z_k$  be defined as follows:  \[ z_k =
2^k + 2^{k-1}  + 2^{k-2} + \cdots +  2+1 = \sum_{i=0}^k2^i =
\frac{2^{k+1}-1}{2-1} = 2^{k+1}-1.\]
Then, by definition, for $j >i$,  $v_{i,j} = 2^{j-i}v_{ii}+z_{j-i-1}$, and
\[
v_{i+1,j} = 2^{j-i-1}v_{i+1,i+1}+z_{j-i-2} = 2^{j-i-1}(3v_{i,i}+2) +z_{j-i-2}.
\]
\[
\frac{3v_{i,j}+1}{2} = \frac{3\times2^{j-i}v_{ii}+3\times z_{j-i-1}+1}{2} = 
3\times2^{j-i-1}v_{ii}+ \frac{3z_{j-i-1}+1}{2}.\]

\[\frac{3z_{j-i-1}+1}{2} = \frac{3(2^{j-i}-1)+1}{2} =  3\times2^{j-i-1}-1 =
2 \times 2^{j-i-1} + (2^{j-i-1}-1) = 2^{j-i}+z_{j-i-2}.\]

Consequently,
\[
\frac{3v_{i,j}+1}{2} = v_{i+1,j}.\] This implies $v_{i+1,j}$
is the odd number that  comes after $v_{i,j}$ in the Collatz
sequence  of $v_{i,j}$. Hence,  for $j  \geq 1$,  the $j$-th
column is the first few  odd numbers in the Collatz sequence
of $u_j$.

\item For $j \geq i  >0$, when $i=j$, we have that $v_{i,j}
   = 3v_{i-1,j-1}+2$,  by definition. Let $i \not  = j$, then
   $v_{i,j} = 2v_{i,j-1}+1$. But $v_{i-1,j-1} \not \equiv 1$
   mod  $4$,   when  $i  \not  =  j$.   Hence  $v_{i,j-1}  =
   (3v_{i-1,j-1}+1)/2$.     Consequently,     $v_{i,j}     =
   3v_{i-1,j-1}+2$.

\item If  $n$ is even, then $v_{0,0}=2^kn+1$  such $k>2$. By
  Part   \ref{3multpart}   of  Proposition   \ref{tailsthm},
  $a_{0,0}  =  2^{k-2}3n+1  \equiv  1$  mod  $3$.  Moreover,
  $a_{0,0}   =   (3v_{0,0}+1)/4$.   Therefore,  by   Theorem
  \ref{timefor2plythm}, $v_{1,1} = 4a_{0,0}+1$. $a_{0,0}$ is
  smaller than  both $v_{0,0}$ and $v_{1,1}$.  Now the proof
  proceeds by  induction. Assume for $k  \geq 0$, $a_{2k,2k}
  \equiv  1$   mod  $3$,  $v_{2k+1,2k+1}=4a_{2k,2k}+1$,  and
  $a_{2k,2k}$   is  smaller   than   both  $v_{2k,2k}$   and
  $v_{2k+1,2k+1}$.  Since   $a_{2k,2k}$  is  odd,   by  part
  \ref{divby2^3part}    of    Proposition    \ref{tailsthm},
  $a_{2k+1,2k+1} = (3v_{2k+1,2k+1}+1)/2^j$ where $j \geq 3$.
  By   Theorem    \ref{timefor2plythm},   $v_{2k+2,2k+2}   =
  2^ja_{2k+1,2k+1}+1$.      Hence,      $a_{2k+2,2k+2}     =
  (3v_{2k+2,2k+2}+1)/4$.       Again,       by       Theorem
  \ref{timefor2plythm}, $v_{2k+3,2k+3}=4a_{2k+2,2k+2}+1$. We
  see that  $a_{2k+2,2k+2}$ is smaller  that $v_{2k+2,2k+2}$
  and $v_{2k+1,2k+1}$. Hence, the proof.

When $n$  is odd, by part  \ref{divby2^3part} of Proposition
\ref{tailsthm},  $a_{0,0}=(3v_{0,0}+1)/2^r$,  such  that  $r
\geq   3$.  Therefore,   by   Theorem  \ref{timefor2plythm},
$v_{1,1}  =  2^ra_{0,0}+1$.  Hence  $a_{1,1}=(3v_{1,1}+1)/4$
which  implies  $v_{2,2}=4a_{1,1}+1$.  Moreover  $a_{1,1}  =
2^{r-2}3a_{0,0}+1  \equiv  1$  mod  $3$. We  also  see  that
$a_{1,1}$ is  smaller than $v_{1,1}$ and  $v_{2,2}$. Now, we
apply a  similar argument,  as in the  $n$ is even  case, to
show that for $k \geq  1$, $a_{2k-1,2k-1} \equiv 1$ mod $3$,
$v_{2k,2k}=4a_{2k-1,2k-1}+1$, and $a_{2k-1,2k-1}$ is smaller
than both $v_{2k-1,2k-1}$ and $v_{2k,2k}$

 \item   By    Part   \ref{v_ijgoesto1mod3part}   and   Part
   \ref{jumppart} of Proposition \ref{tailsthm}, we get that
   If $n$ is odd, then for $k \geq 1$, the Collatz sequences
   of $v_{2k-1,2k-1}$ and $v_{2k,2k}$ are equivalent; and if
   $n$ is even then for $k \geq 0$, the Collatz sequences of
   $v_{2k,2k}$    and   $v_{2k+1,2k+1}$    are   equivalent.
   Consequently, it follows  that if $n$ is odd  for $k \geq
   1$,  the  Collatz sequences  of  $u_{2k-1}$ and  $u_{2k}$
   merge; and if $n$  is even, then for $k  \geq 0$, the Collatz
   sequences of $u_{2k}$ and $u_{2k+1}$ merge.
  
\item  If  $n$  is  odd,  then  by  Part  \ref{jumppart}  of
  Proposition  \ref{tailsthm},  we   get  that  the  Collatz
  sequences  of $u_0=4n+1$  and  $n$ are  equivalent. If  $n
  \equiv  0$ mod $3$,  since $n  < u_0$,  in this  case, the
  Collatz sequence of $u_0$ merges with the Collatz sequence
  of a  smaller number.  Now, let $n  \equiv 2$ mod  $3$. By
  Part          \ref{notmod3part}         of         Theorem
  \ref{multiply2reversetill4thm},    the   reverse   Collatz
  sequence of  $n$ reaches  a number $a$  such that  $a \not
  \equiv 2$  mod $3$. By  definition of the  reverse Collatz
  sequence, it  is strictly decreasing till  it reaches $a$.
  Hence $a  < u_{0}$. Finally, if  $n$ is even,  we get that
  the  Collatz sequence  of $u_0$  reaches a  smaller number
  congruent     to      $1$     mod     $3$      by     Part
  \ref{v_ijgoesto1mod3part}.

\end{enumerate}
\end{proof}

\begin{exm}{\em Let $u_i$ and $v_{i,j}$ be defined as in Theorem  \ref{networkthm}.  Let $t_{i}$ denote the odd number congruent to $1$ mod $3$  such that the Collatz sequences of $v_{i,i}$ and $t_i$ merge:

$n=0:$

\[\begin{array}{lllllllllllllllllllllllllllllllllllllllll} 
u_i:  & 1 & 3 & 7 & 15 & 31 & 63 & 127 & 255 & 511 & 1023 & 2047 & 4095 & 8191  \\ 
v_{1,i}:& &5 & 11 & 23 & 47 &  95 & 191 & 383 & 767 & 1535 & 3071& 6143 & 12287\\ 
v_{2,i}: & & & 17 & 35 &71 & 143 & 287 &  575 & 1151 & 2303 &4607 & 9215  & 18431\\ 
v_{3,i}: && & &  53 & 107 & 215 & 431 & 863 & 1727 & 3455 & 6911 & 13823 & 27647\\
v_{4,i}: && & & & 161 & 323 & 647 & 1295 & 2591 & 5183 & 10367 &20735 &41471 \\
v_{5,i}&&&&&& 485 & 971 & 1943 & 3887 & 7775 &15551 &31103 &62207 \\
v_{6,i}: &&&&&&& 1457 & 2915 & 5831 &11663 & 23327 &46655 &93311 \\
v_{7,i}: &&&&&&&& 4373 & 8747& 17495 & 34991&  69983 & 139967 \\
\end{array}
\]

\[\begin{array}{lllllllllllllllllllllllllllllllllllllllll} 
v_{i,i}: & 1 & 5 & 17 & 53 & 161 & 485 & 1457 & 4373 & \dots \\
t_i:& 1 & 1 & 13 & 13 & 121 & 121 & 1093 & 1093 \\
\end{array}
\]

$n=3:$
\[\begin{array}{lllllllllllllllllllllllllllllllllllllllll} 

u_i: & 13 &    27 & 55 & 111 & 223 & 447&895& 1791&3583& 7167&14335&28671 \\
v_{1,i}: &  & 41 & 83 & 167 & 335 & 671 & 1343 & 2687 & 5375 & 10751 & 21503 & 43007\\ 
v_{2,i}: & & &  125 & 251 & 503 & 1007 & 2015 & 4031 & 8063 & 16127 & 32255 & 64511 \\
v_{3,i}:&  & &  & 377 & 755 & 1511 &  3023 & 6047 &  12095 & 24191 &  48383 & 96767 \\
v_{4,i}: & &&&&1133 & 2267 & 4535 & 9071 & 18143 & 36287 &  72575 &  145151 \\
v_{5,i}: &  &&&&&   3401 &    6803 & 13607 &  27215 &  54431 &  108863 & 217727 \\ 
v_{6,i}: & &&&&&& 10205& 20411& 40823&81647&163295&326591

\end{array}
\]
}  \end{exm}

\[\begin{array}{lllllllllllllllllllllllllllllllllllllllll} 
v_{i,i}: & 13 & 41 & 125 & 377 & 1133 & 3401 & 10205 & \dots\\
t_i:& 13 & 31 & 31 & 283 & 283 & 2551 & 2551
\end{array}
\]

\begin{corollary}
Every odd number occurs in one of the diagonal arrays defined in Theorem  \ref{networkthm}.
\end{corollary}

\begin{proof}
Given  an  odd number  $u_i  \not  \equiv  2$ mod  $3$,  let
$u_{i-1}  = (u_i-1)/2$. This  sequence eventually  reaches a
number  that is  congruent to  $1$  mod $4$,  because it  is
strictly decreasing.  Moreover, $u_i \not \equiv  2$ mod $3$
for all $i$.  Thus, every $u_i \not \equiv  2$ mod $3$ occur
in   one  of   the  diagonal   arrays  defined   in  Theorem
\ref{networkthm}. Now  consider an odd integer  $a \equiv 2$
mod    $3$.   By    Part   \ref{notmod3part}    of   Theorem
\ref{multiply2reversetill4thm}, the reverse Collatz sequence
of $a$  reaches a  number that is  not congruent to  $2$ mod
$3$. Hence, the result.
\end{proof}

\end{document}